\documentclass[11pt,a4j]{amsart}
\usepackage{amsmath}
\usepackage{amssymb}
\usepackage{amsthm}
\usepackage{here}
\theoremstyle{break}
\newtheorem{thm}{Theorem}
\newtheorem{prop}{Proposition}
\newtheorem{lem}{Lemma}

\newtheorem*{question}{Question}

\newtheorem*{ack}{Acknowledgements}
\theoremstyle{remark}
\newtheorem{remark}{Remark}

\newcommand{\ds}{\displaystyle}

\title[An isomorphy criterion for mod $\ell$ Galois
representations.]{An effective isomorphy criterion for mod $\ell$ 
Galois representations.}  
\author{Yuuki Takai
}  
\date{\today}

\address{Graduate School of Mathematics, Nagoya University, Chikusa-ku, Nagoya 464-8602, Japan}
\email{yuuki-takai@math.nagoya-u.ac.jp}

\begin{document}

\maketitle

\begin{abstract}
In this paper, we consider mod $\ell$ Galois 
representations of $\mathbb{Q}$. In particular, 
we obtain an effective 
criterion to distinguish two semisimple  
$2$-dimensional, odd mod $\ell$ 
Galois representations up to isomorphism. 
%Our estimate is better than the estimate in the 
%$1$-dimensional case obtained by using  
%Heath-Brown's result on Linnik's theorem.
%by analytic number theory.
Serre's conjecture 
(Khare-Wintenberger's theorem), Sturm's theorem, 
and its modification by Kohnen are used in our proof.
\end{abstract}

%\begin{keyword}
%%% keywords here, in the form: keyword \sep keyword
%mod $\ell$ Galois representations \sep 
%modular forms \sep
%Serre's conjecture \sep 
%Sturm's theorem.
%%% PACS codes here, in the form: \PACS code \sep code
%
%%% MSC codes here, in the form: \MSC code \sep code
%%% or \MSC[2008] code \sep code (2000 is the default)
%
%\end{keyword}

% \linenumbers

%\section{}
%\label{}

\section{Introduction.}
In this paper, we consider mod $\ell$ Galois representations 
of $\mathbb{Q}$. In particular, we find an effective 
criterion to distinguish two mod $\ell$ Galois representations.
%We consider the number of Frobenius elements that 
% determine mod $\ell$ Galois representations where $\ell$ is a 
%fixed prime number. 
%The mod $\ell$ Galois representation is a continuous homomorphism
% from absolute Galois group over $\mathbb{Q}$
%(denoted by ${\rm Gal}(\overline{\mathbb{Q}}/\mathbb{Q})$ 
%where $\overline{\mathbb{Q}}$ is an algebraic closure of 
%$\mathbb{Q}$) to automorphic group $GL(V)$ of finite dimensional 
%vector space $V$ over algebraic closed field 
%$\overline{\mathbb{F}}_\ell$ of 
%characteristic $\ell$. 
More precisely, we consider the following problem:
\theoremstyle{theorem}
\newtheorem*{problem}{Problem}
\begin{problem}
 Let $\ell$ be a prime number, $d$ and $N$ be two positive 
integers such that $\ell \nmid N$, and 
${\rm Gal}(\overline{\mathbb{Q}}/\mathbb{Q})$ 
be the absolute Galois group of $\mathbb{Q}$, 
where $\overline{\mathbb{Q}}$ is an algebraic 
closure.
Let $\rho, \rho': 
{\rm Gal}(\overline{\mathbb{Q}}/\mathbb{Q})$
$\to GL_d(\overline{\mathbb{F}}_\ell)$ be two $d$-dimensional 
(semisimple) mod $\ell$ Galois representations  
with Artin conductor (outside $\ell$) dividing $N$ 
(defined in section $2.2$). 
Is there an effectively computable 
constant $\kappa = \kappa(\ell, N)$ satisfying 
the following condition $(\ast)$?
%\begin{itemize}
% \item [$(\ast)$]
% If 
%\begin{eqnarray*}
%\det(1 - \rho({\rm Frob}_p)T) =
%\det(1- {\rho'}({\rm Frob}_p)T)  
%\mbox{ in $\overline{\mathbb{F}}_\ell [T]$}
%\end{eqnarray*}
%for every prime number $p$ such that $p\leq\kappa$ and 
%$p \nmid \ell N$, 
%${\rho}$ is isomorphic to $\rho'$. 
%\end{itemize}
\begin{quote}
% \item [$(\ast)$]
$(\ast)$ If 
\begin{eqnarray*}
\det(1 - \rho({\rm Frob}_p)T) =
\det(1- {\rho'}({\rm Frob}_p)T)  
\mbox{ in $\overline{\mathbb{F}}_\ell [T]$}
\end{eqnarray*}
for every prime number $p$ such that $p\leq\kappa$ and 
$p \nmid \ell N$, 
${\rho}$ is isomorphic to $\rho'$. 
\end{quote}
Here ${\rm Frob}_p$ is a Frobenius element at $p$.
%Moreover if there is such $\kappa$, is 
%$\kappa$ effectively computable?  
%\end{itemize}
\end{problem}
%In this article, we investigate the $2$-dimensional case of this 
%problem. 
%The $1$-dimensional case essentially follows from 
%Linnik's theorem for the least prime in an arithmetic 
%progression \cite{Linnik-1}, \cite{Linnik-2}. 
%For the details, see section $3.1$.
%%According to Linnik, we can take $\kappa=c(\ell N)^L$ with positive
%%constants $L$ and $c$.
%Using the result of Heath-Brown on Linnik's theorem \cite{H-B},
%%showed that $L=5.5$ and $c$ is computable (\cite{H-B}). 
%%(It is also conjectured that $c=1$ and $L=2$.) 
%%Therefore in the $1$-dimensional case 
%we can take $\kappa = c(\ell N)^{5.5}$ where $c$ is an 
%effectively computable constant. 
%On the assumption of GRH, we can take 
%$\kappa = c\phi(\ell N)^2\log^2 (\ell N)$ with an effectively computable
%constant $c$, where $\phi$ is the 
%Euler function.
On the $1$-dimensional case, we can trivially take 
$\kappa=\ell N$.  
Using Burgess' result on the estimate of 
character sums \cite{Burgess}, 
we obtain the better estimate for $\kappa$ as follows:
\begin{eqnarray*}
 \kappa \ll_{r,\varepsilon} 
(\ell N)^{\frac{r+1}{4r}+\varepsilon}
\end{eqnarray*}
for every positive integer $r$ and every positive number 
$\varepsilon$. 
%This means that 
%\begin{eqnarray*}
% \kappa < c(\ell N)^{\frac{r+1}{4r}+\epsilon}
%\end{eqnarray*}
%for every positive integer $r$ and every positive number 
%$\epsilon$ with a positive constant $c$ depending on $r$ and 
%$\epsilon$.  
Ankeny \cite{Ankeny} proved, 
under the assumption of GRH, a sharper 
estimate of character sums.
By Ankeny's result, under the assumption of GRH, 
we obtain
\begin{eqnarray*}
 \kappa \ll (\log (\ell N))^2.
\end{eqnarray*} 
For the details, see section $3.1$.  
%%According to Linnik, we can take $\kappa=c(\ell N)^L$ with positive
%%constants $L$ and $c$.
%Using the result of Heath-Brown on Linnik's theorem \cite{H-B},
%%showed that $L=5.5$ and $c$ is computable (\cite{H-B}). 
%%(It is also conjectured that $c=1$ and $L=2$.) 
%%Therefore in the $1$-dimensional case 
%we can take $\kappa = c(\ell N)^{5.5}$ where $c$ is an 
%effectively computable constant. 
%On the assumption of GRH, we can take 
%$\kappa = c\phi(\ell N)^2\log^2 (\ell N)$ with an effectively computable
%constant $c$, where $\phi$ is the 
%Euler function.
\\ \par

In this paper, we consider the $2$-dimensional case.
%From the above reasons, this case is regarded 
%as a $2$-dimensional analogue of the Linnik's theorem. 
%Unfortunately, the criterion of our 
%result depends on level $N$ and also characteristic $\ell$.
The main result is the following theorem:
\theoremstyle{theorem}
\newtheorem{mainthm}{Theorem}
\begin{mainthm}\label{mainthm}
Let $\ell$ be a prime number and $N$ be a positive integer 
such that $\ell \nmid N$.  
%such that $\ell \nmid N$. 
Let ${\rho},{\rho'}:
{\rm Gal}(\overline{\mathbb{Q}}/\mathbb{Q})$
$\to GL_2(\overline{\mathbb{F}}_\ell)$ be two semisimple
$2$-dimensional Galois representations 
%of conductor $N$.
with Artin conductor dividing $N$. 
Assume that $\rho$ is odd ({\it i.e.,}
$\det(\rho)(c)=-1$ for a complex conjugation $c$).
%\begin{eqnarray*}
% \overline{\rho}: {\rm Gal}(\overline{\mathbb{Q}}/\mathbb{Q})
%  \longrightarrow GL_2(\overline{\mathbb{F}_{\ell}})
%\end{eqnarray*} 
Let 
\begin{eqnarray*}
 \kappa = \kappa(N, \ell) = \left\{\begin{array}{cc}
			     \ds\frac{\ell (\ell^2 - 1)^2}{12}
                NN' \prod_{p|N}\left(1+\ds\frac{1}{p}\right) &  \mathrm{if} \ \ell > 2, \\
		 4NN' \prod_{p|N}\left(1+\ds\frac{1}{p}\right) &
				    \mathrm{if} \ \ell = 2,   \end{array}
			     \right.
\end{eqnarray*}
where $N' = \prod_{p|N, \  p^2 \nmid N}p$.
%\begin{eqnarray*}
%N' = \prod_{p|N, \  p^2 \nmid N}p
%\end{eqnarray*}
If 
\begin{eqnarray*}
\det(1- {\rho}({\rm Frob}_p)T) =
\det(1- {\rho'}({\rm Frob}_p)T)  \quad
\mbox{in $\overline{\mathbb{F}}_\ell [T]$}
\end{eqnarray*}
{\it i.e.,}  
\begin{eqnarray*}
 {\rm Tr}({\rho}({\rm Frob}_p)) = 
{\rm Tr}({\rho'}({\rm Frob}_p)) \quad \mbox{ in 
$\overline{\mathbb{F}}_\ell$}, \\
\det({\rho}({\rm Frob}_p)) =
\det({\rho'}({\rm Frob}_p)) \quad \mbox{ in 
$\overline{\mathbb{F}}_\ell$}
\end{eqnarray*}
for every prime number $p$ satisfying $p\leq\kappa$ and 
$p \nmid \ell N$,
%(in $2$-dimensional case, this is  
%\begin{eqnarray*}
% {\rm Tr}(\overline{\rho}({\rm Frob}_p)) \equiv 
%{\rm Tr}(\overline{\rho'}({\rm Frob}_p)) \mod{\lambda}, \\
%\det(\overline{\rho}({\rm Frob}_p)) \equiv 
%\det(\overline{\rho'}({\rm Frob}_p)) \mod{\lambda}
%\end{eqnarray*}
%%${\rm Tr}(\overline{\rho}({\rm Frob}_p)) \equiv 
%%{\rm Tr}(\overline{\rho'}({\rm Frob}_p)) \pmod{\lambda}$ 
%for any prime number $p$ such that $p\leq\kappa$ and 
%$p \nmid \ell N$,) 
then ${\rho}$ is isomorphic to ${\rho'}$. \par
%In particular, let 
%\begin{eqnarray*}
% \kappa' = \kappa'(N, \ell) = \left\{
%			     \begin{array}{ll} 
%			      \ds\frac{\ell(\ell^2 - 1)^2}{24}
%			       N^3\prod_{p|N}(1- \frac{1}{p^2}) 	     
%			     &  \quad {\rm if} \ \ell > 2, \\
%			  \ds 6 N^3\prod_{p|N}(1- \frac{1}{p^2}) 
%			  & \quad {\rm if} \
%			   \ell = 2 \ {\rm and} \ N > 1, \\
%			  \ds 12 &
%			  \quad {\rm if} \ \ell = 2 \ {\rm and} \ N = 1.  
%			     \end{array} 
%\right.
%\end{eqnarray*}
%Then, if 
%\begin{eqnarray*}
%\det(1- {\rho}({\rm Frob}_p)T) =
%\det(1- {\rho'}({\rm Frob}_p)T)  \quad
%\mbox{in $\overline{\mathbb{F}}_\ell [T]$}
%\end{eqnarray*}
%%${\rm Tr}(\overline{\rho}({\rm Frob}_p)) \equiv 
%%{\rm Tr}(\overline{\rho'}({\rm Frob}_p)) \pmod{\lambda}$ 
%for every prime number $p$ satisfying $p\leq\kappa'$ and $p \nmid \ell N$, 
%${\rho}$ is isomorphic to ${\rho'}$. \par

\end{mainthm}

%We remark that, in our case, we have 
%\begin{eqnarray*}
% \kappa(\ell, N) < c\ell^5 N^3
%\end{eqnarray*}
%for an effectively computable constant $c$. 
%Thus our estimate is better than the estimate of Heath-Brown 
%by analytic number theory 
%(though the stronger condition ``$\rho$ is odd'' is 
%assumed in our result).
%%, where $\rho$ is odd if $\det(\rho)(c)=-1$ 
%%for a complex conjugation).

In our proof, we use the theory of modular forms.
Recently, Khare and Wintenberger proved Serre's 
 conjecture for modularity of Galois representations 
\cite{K-W}. Serre's conjecture is the assertion that 
every odd irreducible  $2$-dimensional mod $\ell$ Galois 
representation arises from a newform.
While every odd reducible $2$-dimensional mod $\ell$ Galois representation 
arises from an Eisenstein series.   
By these facts, we can apply the theory of modular 
forms to analyse such Galois representations.  
We also use Sturm's and Kohnen's theorems 
for mod $\ell$ modular forms (\cite{Sturm}, \cite{Kohnen}). 
Roughly speaking, these theorems are assertions 
that the all Fourier coefficients modulo $\ell$ of modular forms are determined 
by the first few Fourier coefficients modulo $\ell$. 
We obtain the  main theorem by applying Sturm's and Kohnen's
result to modular forms associated with Galois representations.\\ \par

\begin{ack}
The author would like to thank Kazuhiro Fujiwara for 
for his elaborated guidance and invaluable discussion.
Akihiko Gyoja read earlier versions of this paper and his 
comments were extremely helpful.
I would also like to thank Kevin Buzzard for helpful remarks.
%his helpful comments concerning an earlier draft of this paper. 

%The author would like to thank Kazuhiro Fujiwara for 
%for his elaborated guidance and invaluable discussion,
% and also would like to thank Akihiko Gyoja for 
%his helpful comments concerning an earlier draft of this paper. 

%It is pleasure to thank K. Fujiwara and A. Gyoja for 
%their help.
%I would like to thank Kazuhiro Fujiwara who 
%brought this subject to my attention. 
%I would also like to
% thank Akihiko Gyoja for 
%his helpful comments concerning an earlier draft of this paper. 
\end{ack}

\section{Preliminaries.}

%\section*{Notation and conventions.}
{\bf Notation and conventions.}
In this paper, we follow the notations and definitions in \cite{Diamond-Shurman}.  
For the details of modular forms, we also refer to \cite{Shimura} or \cite{Miyake}. 

% Let $k,N$ be positive integers. 
% $\Gamma(N)$ (resp. $\Gamma_1(N)$, and resp. $\Gamma_0(N)$) 
%denotes 
% congruence subgroup of $\Gamma(1)SL_2(\mathbb{Z})$ 
% consisting of matrices 
% ${\footnotesize \left( \begin{array}{cc} a & b \\ c&d \end{array}\right)}$
% such that $c\equiv 0 \pmod{N}$ (resp. 
%$a\equiv d \equiv 1 \pmod{N}$ and $c\equiv 0 \pmod{N}$, 
%and resp. $a\equiv d \equiv 1 \pmod{N}$ and 
%$b\equiv c\equiv 0 \pmod{N}$).
Let $k$ and $N$ be positive integers. 
 Congruence subgroup $\Gamma_0(N)$ of $SL_2(\mathbb{Z})$ is  
 defined as follows:
       \begin{eqnarray*}
%%	\Gamma(N)= {\footnotesize \left\{ \left( \begin{array}{cc} a & b \\ c & d
%%				   \end{array}
%			    \right) \in SL_2(\mathbb{Z}) \left|
%			    \left( \begin{array}{cc} a & b \\ c & d
%				   \end{array}
%		     \right) \equiv 
%			    \left( \begin{array}{cc} 1 & 0 \\ 0 & 1
%				   \end{array}
%		     \right) \mod{N}				  
%							 \right.
%		    \right\}}, \\
%	\Gamma_1(N)= {\footnotesize \left\{ \left( \begin{array}{cc} a & b \\ c & d
%				     \end{array}
%		     \right) \in SL_2(\mathbb{Z}) \left|
%		\left( \begin{array}{cc} a & b \\ c & d
%			    \end{array}
%		     \right) \equiv 
%		\left( \begin{array}{cc} 1 & \ast \\ 0 & 1
%		       \end{array}
%		     \right) \mod{N}				  
%						  \right.
%	     \right\}}, \\
	\Gamma_0(N)= {\footnotesize \left\{ \left( \begin{array}{cc} a & b \\ c & d
				     \end{array}
		     \right) \in SL_2(\mathbb{Z}) \left|
		\left( \begin{array}{cc} a & b \\ c & d
			    \end{array}
		     \right) \equiv 
		\left( \begin{array}{cc} \ast & \ast \\ 0 & \ast
		       \end{array}
		     \right) \mod{N}				  
						  \right.
	     \right\}}.
       \end{eqnarray*}
We remark that $\Gamma_0(1)=SL_2(\mathbb{Z})$.
Let $\mathcal{H}$ be the complex upper half plane 
and $GL_2^+(\mathbb{R})$ be the subgroup of $GL_2(\mathbb{R})$
consisting of the elements with positive determinant.  
For a holomorphic function $f$ on $\mathcal{H}$, 
      the action of 
      $\gamma= \left( \begin{array}{cc} a & b \\ c & d
			   \end{array}
		     \right)\in GL_2^+(\mathbb{R})$ 
		     on 
		     $f$  is defined as follows:
       \begin{eqnarray*}
	(f |[\gamma])(z) = 
	 \det(\gamma)^{\frac{k}{2}}(cz + d)^{-k}f(\gamma
	 z).
       \end{eqnarray*}  
%For $\Gamma=\Gamma(N),\Gamma_1(N),\Gamma_0(N)$, 
%the complex vector space of cusp forms of weight $k$ and level
%$\Gamma$ is denoted by $S_k(\Gamma)$. 
Let $\chi$ be a mod $N$ Dirichlet character.
The complex vector space of the holomorphic modular 
forms of weight $k$ and level $\Gamma_0(N)$ with 
Nebentypus $\chi$ is denoted by $M_k(\Gamma_0(N), \chi)$
, {\it i.e.,} 
\begin{eqnarray*}
 && M_k(\Gamma_0(N),\chi) = \\  && \ \ \ \ \ \ \ \ \ \{ f:\mathcal{H} \to
  \mathbb{C} \ | \   
  \mathrm{holomorphic \ on \ } \mathcal{H} \mathrm{\ and \ the \
   cusps}, \ f|[\gamma] = \chi(d)f \} 
% && M_k(\Gamma_0(N),\chi) = \\ && \{ f:\mathcal{H} \to \mathbb{C} \ \ | \ 
%  \mathrm{holomorphic},  \ f|[\gamma] = \chi(d)f, \mathrm{ \ satisfying
%  \ the \ cusp \ conditions.} \}
\end{eqnarray*}
%  Let $\Gamma = \Gamma(N), \Gamma_1(N),\Gamma_0(N)$,
%and 
%	\begin{eqnarray*}
%	 h= \left\{ 
%	      \begin{array}{ll}
%	      N  &\quad \mbox{if $\Gamma = \Gamma(N)$},
%	      \\
%	      1 &\quad \mbox{if 
%	       $\Gamma = \Gamma_1(N),\Gamma_0(N)$}.
%	     \end{array}
%	     \right.
%	\end{eqnarray*} 
%the least positive integer such that 
%      $\ds\left(\begin{array}{cc}
%      1 & h \\ 0 & 1	\end{array}\right) \in \Gamma$.
      Then $f\in M_k(\Gamma_0,\chi)$ has its 
      Fourier expansion $f=\sum_{n\geq 0} a_n q^n$ where 
      $q=e^{2\pi i z} \ (z\in \mathcal{H})$. 

\subsection{Modular forms.}

\subsubsection{Operator $\pi$.}
Here we construct an operator on the space of modular forms. 
For the details, we refer to \cite[Lemma~4.6.5]{Miyake}.

Let $k$ and $N$ be positive integers, $\chi$ be a Dirichlet character
mod $N$, and 
% and $S_k(\Gamma(N))$ be a set of cusp forms of weight $k$ and level 
%$\Gamma(N)$.  
$f(z)= \sum_{n=0}^\infty a_n q^n\in M_k(\Gamma_0(N), \chi) $. For a
positive integer $d$, we define the operator 
$V(d):M_k(\Gamma_0(N), \chi)\to M_k(\Gamma_0(dN), \chi)$ as follows:
\begin{eqnarray*}
 (f | V(d))(z) = \left(f\left| \left[ \begin{array}{cc}
				d & 0 \\ 0 & 1 
				       \end{array}
				\right]      \right.\right) (z) =  \sum_{n=0}^{\infty} a_n q^{dn}.
\end{eqnarray*}
For a prime number $p$ satisfying $p | N$, 
we define operator $U(p):M_k(\Gamma_0(N), \chi)$ $\to $  
 $M_k(\Gamma_0(pN),\chi)$ as follows (cf. \cite[Lemma~4.6.5]{Miyake}): 
\begin{eqnarray*}
 (f | U(p))(z) = \frac{1}{p}\sum_{m=0}^{p-1}\left(f\left| \left[ \begin{array}{cc}
				1 & m \\ 0 & p 
				       \end{array}
				\right]      \right.\right) (z) =  \sum_{n=0}^{\infty} a_{pn} q^{n}.
\end{eqnarray*}
It is easy to show that if $N=p^rN_0$ such that $r\geq 2$ and 
$(p, N_0)=1$, and $\chi$ is of mod $p^{r-1}N_0$, then   
$U_p:M_k(\Gamma_0(p^rN_0), \chi)$ $\to$ $M_k(\Gamma_0(p^{r-1}N_0),\chi)$.
For a prime number $p$ dividing $N$, we set operator 
$\pi_p=V(p)\circ U(p):M_k(\Gamma_0(N), \chi)\to M_k(\Gamma_0(M), \chi)$, 
where 
\begin{eqnarray*}
M = \left\{
\begin{array}{cc}
pN  & \mathrm{if} \ p^2 \nmid N,   \\
  N  & \mathrm{otherwise.}   
\end{array}          \right.
\end{eqnarray*}
Then, we have 
\begin{eqnarray*}
 \pi_p (f) (z) = \sum_{n=0}^{\infty} a_{pn}q^{pn}.
\end{eqnarray*}
%Let $N'=\prod_{p|N:\mathrm{prime}}p$. 
The operator $\pi:M_k(\Gamma_0(N), \chi)\to M_k(\Gamma_0(NN'), \chi) $ is 
defined as follows:
\begin{eqnarray*}
 \pi = \mbox{id} - \sum_{p|N:{\rm prime}} \pi_p + 
  \sum_{p_1,p_2|N:{\rm prime}} \pi_{p_1}\circ \pi_{p_2} - \cdots 
      = \prod_{p|N:{\rm prime}}(\mbox{id} - \pi_p),
\end{eqnarray*}
where 
\begin{eqnarray*}
N' = \prod_{p|N, \  p^2 \nmid N}p
\end{eqnarray*}
 and $\mbox{id}$ is the identity of $M_k(\Gamma_0(NN'), \chi)$, 
and the product in the last equation is taken as operators.
Then, we have 
\begin{eqnarray*}
 \pi(f)(z) = \sum_{n \geq 1 : (n,N) = 1} a_n q^n.
\end{eqnarray*}

%We define $\Gamma^1(N)$ as follows:
%\begin{eqnarray*}
%	 \Gamma^1(N)= 
%	 {
%	 \left. \left\{ \left(
%  \begin{array}{cc}
%   a & b \\
%   c & d 
%  \end{array}
%  \right)\in SL_2(\mathbb{Z}) \ \right|  \ 
%\left(
%  \begin{array}{cc}
%   a & b \\
%   c & d 
%  \end{array}
%  \right) \equiv 
%\left(
%  \begin{array}{cc}
%   1 & 0 \\
%   \ast & 1 
%  \end{array}
%  \right) \pmod{N}
%%a,d \in 1 + 
%%	N\mathcal{O}_F, b\in N(\mathfrak{ab})^*, %\mathfrak{N}^{-1},
%%	c \in N(\mathfrak{ab}\mathfrak{d}_F) 
%	 \right\}
%	 }.
%\end{eqnarray*}

\subsubsection{Sturm's and Kohnen's result.}
We recall Sturm's and Kohnen's result 
(in a form we need). 
These are results on the number of the first 
Fourier coefficients that determine the all Fourier coefficients (mod $\ell$)
of modular forms.

Let $K$ be a number field, and $\mathcal{O}_K$ be the ring of integers
of $K$.
\begin{lem}[Sturm {\cite[Th.~1]{Sturm}}]
 Let $k$ and $N$ be positive integers, and $\chi$ be a mod $N$ Dirichlet
 character. 
%and $\Gamma=\Gamma(N)$. 
For $f, g\in M_k(\Gamma_0(N),\chi)$, 
\begin{eqnarray*}
 f(z) &=& \sum_{n=0}^\infty a_n q^n, \\ 
 g(z) &=& \sum_{n=0}^\infty b_n q^n
\end{eqnarray*}
denote their Fourier expansions. 
%Let $K$ be a number field, and $\mathcal{O}_K$ be the ring of
% integers in $K$. 
Let $\ell$ be a prime number, and $\lambda$ be a prime ideal of 
$\mathcal{O}_K$ such that $\lambda | \ell\mathcal{O}_K$.
 We assume that $a_n, b_n$, and the values of $\chi$ are in $\mathcal{O}_K$ for all $n$. 
If 
%$f \not \equiv g \pmod{\lambda}$, 
%there exists 
$a_n \equiv b_n \pmod{\lambda}$ for every $n$ such that 
\begin{eqnarray*}
 n \leq \frac{k}{12}[\Gamma_0(1):\Gamma_0(N)],
\end{eqnarray*}
then $a_n \equiv b_n \pmod{\lambda}$ for all $n$.
\end{lem}
%\begin{proof}
% See \cite{Sturm}.
%\end{proof}

\begin{lem}[Kohnen {\cite[Th.~1]{Kohnen}}]
 Let $k_1$ and $k_2$ be two positive integers such that $k_1,k_2\geq 2$
 and $k_1 \not = k_2$, 
 $N$ be a positive integer, and $\chi$ be a mod $N$ Dirichlet
 character.
For $f \in M_{k_1}(\Gamma_0(N),\chi)$ and $g\in M_{k_2}(\Gamma_0(N),\chi)$,
\begin{eqnarray*}
 f(z) &=& \sum_{n=0}^\infty a_n q^n, \\ 
 g(z) &=& \sum_{n=0}^\infty b_n q^n
\end{eqnarray*}
 denote their Fourier expansions.
%Let $K$ be a number field, and $\mathcal{O}_K$ be the ring of
% integers in $K$. 
Let $\ell$ be a prime number, and $\lambda$ be a prime ideal of 
$\mathcal{O}_K$ such that $\lambda | \ell\mathcal{O}_K$.
 We assume that $a_n, b_n$, and the values of $\chi$ are in
 $\mathcal{O}_K$ for all $n$.
%\in \mathcal{O}_K$ for all $n$. 
If 
%$f \not \equiv g \pmod{\lambda}$, there exists 
$a_n\equiv b_n \pmod{\lambda}$ for every $n$ such that 
\begin{eqnarray*}
 n \leq \frac{{\rm max}\{k_1,k_2\}}{12}\left\{
[\Gamma_0(1):\Gamma_0(N)\cap\Gamma_1(\ell)] \ \ \ \mathrm{ if} \ 
\ell > 2, \atop
[\Gamma_0(1):\Gamma_0(N)\cap\Gamma_1(4)] \ \ \ \mathrm{ if} \ \ell = 2,
\right.
\end{eqnarray*}
then $a_n \equiv b_n \pmod{\lambda}$ for all $n$.
%$a(n)\not \equiv b(n) \pmod{\lambda}$.
\end{lem}
%\begin{proof}
% See \cite{Kohnen}.
%\end{proof}

\begin{remark}\label{remark:weight}
We remark that Lemma $2$ holds even for $k_1=k_2$.
Indeed, because 
$[\Gamma_0(1):\Gamma_0(N)] \leq [\Gamma_0(1):\Gamma_0(N)\cap\Gamma_1(\ell)]$, 
applying Sturm's theorem if $k_1=k_2$, we can show that 
the first 
\begin{eqnarray*}
 \frac{{\rm max}\{k_1,k_2\}}{12}
[\Gamma_0(1):\Gamma_0(N)\cap\Gamma_1(\ell)]
\end{eqnarray*} 
coefficients determine the all Fourier coefficients 
(mod $\lambda$) of modular forms for 
every $k_1,k_2\geq 2$ and $\ell>2$. 
Similarly, also in the case of $\ell =2$, the first 
\begin{eqnarray*}
  \frac{{\rm max}\{k_1,k_2\}}{12}
[\Gamma_0(1):\Gamma_0(N)\cap\Gamma_1(4)]
\end{eqnarray*}
coefficients determine the all Fourier coefficients 
(mod $2$) of modular forms for every $k_1,k_2\geq 2$.
\end{remark}

%\begin{remark}
% We remark that the assumption of integrality
% in Lemma $1$ and Lemma $2$ is only for Fourier coefficients 
%\end{remark}

\subsection{Galois representations.}

\subsubsection{$\ell$-adic and mod $\ell$ Galois representations.}
Let $\ell$ be a prime number, $d$ be a positive integer, 
and $L$ be a finite extension of $\mathbb{Q}_\ell$.
Then a $d$-dimensional $\ell$-adic Galois representation 
over $L$ is a continuous homomorphism 
\begin{eqnarray*}
 \rho : {\rm Gal}(\overline{\mathbb{Q}}/\mathbb{Q}) \longrightarrow GL_d(L).
\end{eqnarray*}
%By Chebotarev density theorem, Frobenius elements ${\rm Frob}_p$ for 
%all but finitely many primes $p$ determine $\rho$. 
%Let $\rho,\rho'$ be two $d$-dimensional  
%$\ell$-adic Galois representations over $L$.
For two $d$-dimensional $\ell$-adic representations 
$\rho$ and $\rho'$, $\rho$ is isomorphic (or equivalent) to $\rho'$
(denoted by $\rho\simeq \rho'$)
if there is an element $A\in GL_d(L)$ such that 
$\rho(\sigma) = A^{-1}\rho'(\sigma)A$ for all 
$\sigma\in {\rm Gal}(\overline{\mathbb{Q}}/\mathbb{Q})$. 
%$\rho$ is reducible if there is a subspace 
%$\{0\}\subsetneq V\subsetneq L^d$ such that 
%$\rho(\sigma)V\subset V$ for all 
%$\sigma\in {\rm Gal}(\overline{\mathbb{Q}}/\mathbb{Q})$. 
%$\rho$ is irreducible if there is no subspace 
%$\{0\}\subsetneq V\subsetneq L^d$ such that 
%$\rho(\sigma)V\subset V$ for all 
%$\sigma\in {\rm Gal}(\overline{\mathbb{Q}}/\mathbb{Q})$. 
$\rho$ is absolutely irreducible if 
for every finite extension $L'$ of $L$,
the composition representation $f\circ \rho$ is 
also irreducible, where $f:GL_d(L) \hookrightarrow GL_d(L')$. 
$\rho$ is odd if ${\rm det}(\rho(c))=-1$ for a 
complex conjugation 
$c\in {\rm Gal}(\overline{\mathbb{Q}}/\mathbb{Q})$.

Let $\mathbb{F}$ be a finite field or an algebraic 
closed field of characteristic $\ell$.
A $d$-dimensional mod $\ell$ Galois representation 
 over $\mathbb{F}$ is a continuous homomorphism 
$\rho:{\rm Gal}(\overline{\mathbb{Q}}/\mathbb{Q})$
$\to GL_d(\mathbb{F})$, where $GL_d(\mathbb{F})$ has a 
discrete topology. 
The notions of isomorphic, 
absolutely irreducible, and odd are defined as above.
%When $\mathbb{F}$ is an algebraic closed field, 
%because 
Since ${\rm Gal}(\overline{\mathbb{Q}}/\mathbb{Q})$ is 
compact, $\rho$ has a finite image. 
Therefore, when $\mathbb{F} = \overline{\mathbb{F}}_\ell$, there is a finite subfield 
$\mathbb{F}'$ of $\mathbb{F}$ such that $\rho$ is defined over
$\mathbb{F}'$.
Let $\mathbb{F}$ be a field of characteristic $\ell$, and 
$V=\mathbb{F}^d$. 
For a $d$-dimensional mod $\ell$ representation 
$\rho:{\rm Gal}(\overline{\mathbb{Q}}/\mathbb{Q}) \to GL(V)$,
Artin conductor (outside $\ell$) 
$N(\rho)$ of $\rho$ is defined as follows:
\begin{eqnarray*}
 N(\rho) = \sum_{p\not = \ell: {\rm prime}} p^{n_p(\rho)}, \ 
n_p(\rho) = \sum_{i\geq 0} \frac{1}{(G_{p,0}: G_{p,i})}
\dim (V/V^{G_{p,i}}), 
\end{eqnarray*}
where $G_{p,i}$ is the $i$-th ramification group of the 
decomposition group at $p$. 
%of a extension of $p$ to 
%a splitting field $\overline{\mathbb{Q}}^{{\rm Ker}\rho}$. 
It is known that 
$n_p(\rho)$ is a non-negative integer. 
Thus $N(\rho)$ is also a positive integer.
We remark that the Artin conductor 
is relatively prime to $\ell$.

There is the fundamental fact on isomorphy of 
Galois representations.
\begin{lem}
 Let $R$ be a finite extension field of $\mathbb{Q}_\ell$ or 
a finite extension field of $\mathbb{F}_\ell$. 
Let $\rho,\rho': {\rm Gal}(\overline{\mathbb{Q}}/\mathbb{Q})$
$\longrightarrow GL_d(R)$ be two semisimple continuous Galois 
representations. Then if 
$\det(1 - \rho({\rm Frob}_p)T) = \det(1-\rho'({\rm Frob}_p)T) \in R[T]$ 
for all but finitely many prime number $p$, 
$\rho$ is isomorphic to $\rho'$.
\end{lem}
\begin{proof}
The lemma follows from Chebotarev's density theorem and Brauer and
 Nesbitt's theorem (cf. \cite[Lem.~3.2]{Deligne-Serre}).
\end{proof}

Next, we discuss mod $\ell$ reductions of 
$\ell$-adic representations.
Let $L$ be a finite extension of $\mathbb{Q}_\ell$, 
$\mathcal{O}_L$ be the ring of integers in $L$, 
and $\lambda$ be a maximal ideal in $\mathcal{O}_L$ such that 
$\lambda|\ell\mathcal{O}_L$.
Let $\rho: {\rm Gal}(\overline{\mathbb{Q}}/\mathbb{Q})$
$\to GL(V)$ be a $d$-dimensional  
$\ell$-adic Galois representation, where $V$ is a 
$d$-dimensional vector space over $L$. 
Then the following fact is known:
\begin{prop}
% Let $\rho$ be a $d$-dimensional  
%$\ell$-adic Galois representations over $L$. 
%Then $\rho$ is similar to a Galois representation 
%$\rho': {\rm Gal}(\overline{\mathbb{Q}}/\mathbb{Q})$
%$\longrightarrow GL_d(\mathcal{O}_L)$ where $\mathcal{O}_L$ is 
%the ring of algebraic integers in $L$.
$\rho$ admits a Galois stable lattice
{\it i.e.,} 
there is a lattice $\mathcal{L} \subset V$ such that 
$\mathcal{L}$ is stable under 
${\rm Gal}(\overline{\mathbb{Q}}/\mathbb{Q})$.
\end{prop}
\begin{proof}
 For the proof, cf. \cite[Prop.~9.3.5]{Diamond-Shurman}.
\end{proof}

%For a finite extension field $L$ of $\mathbb{Q}_\ell$, 
%there is a finite 
%extension field $K$ of $\mathbb{Q}$ and there is a prime ideal 
%$\lambda$ in $\mathcal{O_K}$ such that $\lambda|\ell\mathcal{O}_K$ 
%such that $L\simeq K_\lambda$ where $K_\lambda$ is 
%the completion of $K$ at $\lambda$.
%A finite extension field $L$ of $\mathbb{Q}_\ell$ is 
%isomorphic to 
%$K_\lambda$ where $K$ is a finite extension field of 
%$\mathbb{Q}$, $\lambda$ is a prime ideal 
%in $\mathcal{O}_K$ such that $\lambda|\ell\mathcal{O}_K$
%and $K_\lambda$ is a completion of $K$ at $\lambda$.

By choosing some basis which generates 
a Galois stable lattice in $V$ for $\rho$, 
we can take $GL_d(\mathcal{O}_L)$ as the image of $\rho$.
%The $d$-dimensional mod $\ell$ Galois representation $\overline\rho$
%associated to $\rho$ is 
A reduction $\overline\rho$ of $\rho$ is defined as  
%$\overline{\rho}: {\rm Gal}(\overline{\mathbb{Q}}/\mathbb{Q})$
%$\longrightarrow GL_d(\mathcal{O}_{L}/\lambda)$ by the 
the composition $f\circ \rho$, where $f$ is a surjective continuous homomorphism 
$GL_d(\mathcal{O}_L)\to$ 
$GL_d(\mathcal{O}_L/\lambda)$. 
We remark that the semisimplification of $\overline\rho$ 
does not depend on any choice of Galois stable lattices.
%This $\overline\rho$ is called by a $d$-dimensional 
%mod $\ell$ Galois representation.  

%Next, we recall $\ell$-adic Galois representations 
%associated to newforms.
%Let $f\in S_k(\Gamma_0(N),\epsilon)$ be a normalised newform with Nebentypus $\epsilon$
%(for the definition, cf. \cite[Def.~5.8.1]{Diamond-Shurman}) and 
%$f = \sum a_n q^n (q = e^{2\pi i z})$ 
%denotes the Fourier expansion of $f$. 
%Let $K = \mathbb{Q}(\dots, a_n, \dots)$ be the field generated
%by the all Fourier coefficients of $f$.
%Then it is known that $K$ is a finite extension of 
%$\mathbb{Q}$ and coefficients $a_n$ are 
%in $\mathcal{O}_K$ (cf. \cite[Theorem $3.48$]{Shimura}). 
%Let $\ell$ be a prime number such that $\ell\nmid N$ and 
%$\lambda$ be a prime ideal of $\mathcal{O}_K$ such that 
%$\lambda|\ell\mathcal{O}_K$. 
%The $\ell$-adic Galois representation associated to $f$ is 
%a $2$-dimensional representation 
%$\rho = \rho_f:{\rm Gal}(\overline{\mathbb{Q}}/\mathbb{Q})$
%$\to GL_2(\mathcal{O}_{K_\lambda})$ such that 
%\begin{eqnarray*}
% {\rm Tr}(\rho({\rm Frob}_p)) &=& a_p, \\
% \det (\rho({\rm Frob}_p)) &=& \epsilon(p)p^{k-1}
%\end{eqnarray*}
%for all prime number $p$ satisfying $(p,\ell N)=1$, where 
%$K_\lambda$ is the completion of $K$ at $\lambda$.
%A mod $\ell$ Galois representation associated to 
%$f$ is denoted by $\overline{{\rho}_f}$.

\subsubsection{Modularity of 2-dimensional mod $\ell$ Galois
   representations.}
In this section, we discuss the theory of between modular forms and
$\ell$-adic and mod $\ell$ Galois representations.
In the beginning, we recall $\ell$-adic Galois representations 
associated to newforms.
Let $f\in S_k(\Gamma_0(N),\epsilon)$ be a normalised newform with Nebentypus $\epsilon$
(for the definition, cf. \cite[Def.~5.8.1]{Diamond-Shurman}) and 
$f = \sum a_n q^n (q = e^{2\pi i z})$ 
denotes the Fourier expansion of $f$. 
Let $K = \mathbb{Q}(\dots, a_n, \dots, \epsilon)$ 
be the field generated
by the all Fourier coefficients of $f$ and the values of $\epsilon$.
Then it is known that $K$ is a finite extension of 
$\mathbb{Q}$ and coefficients $a_n$ are 
in $\mathcal{O}_K$ (cf. \cite[Theorem $3.48$]{Shimura}). 
Let $\ell$ be a prime number such that $\ell\nmid N$ and 
$\lambda$ be a prime ideal of $\mathcal{O}_K$ such that 
$\lambda|\ell\mathcal{O}_K$. 
The $\ell$-adic Galois representation associated to $f$ is 
a $2$-dimensional representation 
$\rho = \rho_f:{\rm Gal}(\overline{\mathbb{Q}}/\mathbb{Q})$
$\to GL_2(\mathcal{O}_{K_\lambda})$ such that 
\begin{eqnarray*}
 {\rm Tr}(\rho({\rm Frob}_p)) &=& a_p, \\
 \det (\rho({\rm Frob}_p)) &=& \epsilon(p)p^{k-1}
\end{eqnarray*}
for all prime number $p$ satisfying $(p,\ell N)=1$, where 
$K_\lambda$ is the completion of $K$ at $\lambda$.
A mod $\ell$ Galois representation associated to 
$f$ is denoted by $\overline{{\rho}_f}$. 
We also define  
$\ell$-adic and mod $\ell$ Galois representations associated to 
Eisenstein series which are normalised eigenforms as well as 
newforms (cf. \cite[Th.~9.6.6]{Diamond-Shurman}). \\ \par

Next, We recall some facts of modularity of 2-dimensional mod $\ell$ 
Galois representations. Let 
$\rho:{\rm Gal}(\overline{\mathbb{Q}}/\mathbb{Q})\to$ 
$ GL_2(\overline{\mathbb{F}}_\ell)$ be a mod $\ell$ Galois representation. 
We assume that $\rho$ is semisimple, odd, and $N(\rho)|N$.
%and that Artin conductor $N(\rho)$ divides $N$. 

When $\rho$ is reducible, $\rho$ comes from an Eisenstein series. 
More precisely, we explain modularity in the reducible case. 
At first, we review some facts on Eisenstein series (for the details,
cf. \cite[Sect.~4.5]{Diamond-Shurman}). 
Let $N$, $u$ and $v$ be positive integers such that $uv|N$,  
$\psi$ (resp. $\phi$) be a (resp. primitive) mod $u$ (resp. $v$) 
Dirichlet character, 
and $k$ be an integer $k\geq 2$ such that $(\psi\phi)(-1)=(-1)^k$.
$E_k^{\psi,\phi}$ denotes the Eisenstein series defined in
\cite[Sect.~4.5, Sect.~4.6]{Diamond-Shurman}. 
We remark that $E_k^{\psi,\phi}$ is holomorphic if $k\geq 3$, 
but is not holomorphic if $\psi$ and $\phi$ are principal and $k=2$.  
%We define Eisenstein series $E_k^{\psi,\phi}$ by the following
%$q$-expansion (cf. \cite[Th.~4.5.1]{Diamond-Shurman}):
$E_k^{\psi,\phi}$ has the following Fourier expansion: 
\begin{eqnarray*}
 E_k^{\psi,\phi}(z) =  \delta(\psi)L(1-k, \phi) + 
  2\sum_{n=1}^{\infty}\sigma_{k-1}^{\psi,\phi}(n)q^n,  
\end{eqnarray*}
where 
\begin{eqnarray*}
 \delta(\psi) = \left\{ \begin{array}{ll}
		1 & \mathrm{if} \ \psi 
		 \mathrm{ \ is \ a \ primitive \ character}, \\
		         0 & \mathrm{otherwise,} 
	\end{array}
\right. 
\end{eqnarray*}
and 
\begin{eqnarray*}
 \sigma_{k-1}^{\psi,\phi}(n) = \sum_{d|n,
  d>0}\psi\left(\frac{n}{d}\right)\phi(d) d^{k-1}.
\end{eqnarray*}
For a positive integer $t$ such that $tuv|N$, we set 
\begin{eqnarray*}
 E_k^{\psi,\phi,t}(z)= \left\{
                      \begin{array}{ll}
                        E_k^{\psi,\phi}(z) - t E_k^{\psi,\phi}(tz) 
			 & \mathrm{if} \ \psi \mathrm{\ and \ } \phi
			 \mathrm{\ are \ primitive \ and \ } k=2,   \\
		        E_k^{\psi,\phi}(tz) &  \mathrm{otherwise.}
                      \end{array}
%E_k^{\psi,\phi}(z) - t E_k^{\psi,\phi}(tz) \
%			 \ \mathrm{if} \ \psi=1, \phi=1, k=2 \atop 
%			 E_k^{\psi,\phi}(tz) \ \ \ \ \ \ \ \ \ \ \ \ \ \
%			 \ \ \ \mathrm{otherwise.} 
\right.
\end{eqnarray*}
Then, $E_k^{\psi,\phi,t}(z)$ is a holomorphic modular form.
If $uv=N$ or $k=2$, $\psi=1$, $\phi=1$, $t$ is a prime and $N$ is a
power of $t$, then $ E_k^{\psi,\phi,t}$ is an eigenform for all Hecke operators.
Let $\rho$ be as above. We assume that $\rho$ is reducible. Then, we can
show that 
\begin{eqnarray*}
 \rho \simeq \left(\begin{array}{cc}
	          %  \varepsilon \chi_\ell^a & 0 \\ 
%		           0         & \varepsilon' \chi_\ell^b
  \psi \chi_\ell^a & 0 \\ 
		           0         & \phi \chi_\ell^b
 		   \end{array}\right), 
\end{eqnarray*}
where $\psi$ and $\phi$ are Dirichlet characters such that $\psi\phi$ is
a mod $N$ Dirichlet character, 
$\chi_\ell$ is the cyclotomic character, and $a$ and $b$ are 
integers such that $0\leq a,b \leq \ell - 2$ (cf. \S.3.1). 
When $a=0$, $\rho$ comes from $\frac{1}{2}E_k^{\psi,\phi,t}$ for 
$2\leq k \leq \ell+ 1$ and $k \equiv b \mod{\ell - 1}$ (cf. \cite[Th.~9.6.7]{Diamond-Shurman}).
When $a=1, \dots, \ell -2$, for a positive integer $k$ such that 
$k\equiv b-a+1 \mod{\ell - 1}$ and $2\leq k \leq l+1$, 
$\rho$ comes from eigenform $\frac{1}{2}\theta^aE_{k}^{\psi,\phi,t}$. 
Here $\theta$ is the theta operator (cf. \cite[Sect.~3]{edixhoven1992weight}). 
We remark that the filtration $w(\frac{1}{2}\theta^aE_{k}^{\psi,\phi,t})$ of
$\frac{1}{2}\theta^aE_{k}^{\psi,\phi,t}$ is 
$2\leq w(\frac{1}{2}\theta^aE_{k}^{\psi,\phi,t})\leq \ell^2 - 1$ 
if $a = 1,2,\dots,\ell - 2$.

When $\rho$ is irreducible, 
Khare and Wintenberger proved the following theorem known as Serre's
modularity conjecture:

\theoremstyle{theorem}
\newtheorem{serreconj}{Theorem}
\renewcommand{\theserreconj}{}

\begin{serreconj}[Khare and Wintenberger {\cite{K-W}}]%, Khare-Wintenberger's Theorem]
 Let $\rho:$ ${\rm Gal}(\overline{\mathbb{Q}}/\mathbb{Q})$
$\to GL_2(\overline{\mathbb{F}}_\ell)$ be an  
irreducible odd $2$-dimensional mod $\ell$ Galois representation.
Then there is a newform 
$f\in S_{k(\rho)}(\Gamma_0(N(\rho)), \epsilon(\rho))$ 
such that $\rho$ is isomorphic to $\overline{\rho_f}$.
Here the integers $N(\rho),k(\rho)$, and the mod $N(\rho)$ Dirichlet character 
$\epsilon(\rho)$ are defined by Serre \cite[Sect.~1 and Sect.~2]{Serre}. 
\end{serreconj}
%\begin{proof}
% The theorem is proved by Khare and Wintenberger \cite{K-W}.
%\end{proof}

%\begin{remark}
% $N(\rho)$ is the (prime to $\ell$) Artin conductor of $\rho$. 
%By the definitions in \cite{Serre}, 
%$2\leq k(\rho)\leq \ell^2 - 1$ if $\ell >2$ and 
%$k(\rho)=2$ or $4$ if $\ell = 2$.
%\end{remark}
\begin{remark}\label{remark:modularity}
% $N(\rho)$ is the (prime to $\ell$) Artin conductor of $\rho$. 
By the definitions in \cite[Sect.~2]{Serre}, 
$2\leq k(\rho)\leq \ell^2 - 1$ if $\ell >2$ and 
$k(\rho)=2$ or $4$ if $\ell = 2$. 
We remark that both the reducible case and the irreducible case, 
we can take a modular form corresponding to $\rho$ of weight $k$ 
such that $2\leq k \leq \ell^2-1$ if $\ell>2$ and $k=2$ or $4$ if 
$\ell = 2$. 
\end{remark}

\section{One and two-dimensional cases.}
\subsection{One-dimensional case.}
Let $\ell$ be a prime number, and  
$N$ be a positive integer such that $\ell \nmid N$.
Let $\rho: $ ${\rm Gal}(\overline{\mathbb{Q}}/\mathbb{Q})$
$\to GL_1(\overline{\mathbb{F}}_\ell)$ be a $1$-dimensional 
mod $\ell$ Galois representation of the Artin conductor $N$.
%Since $\rho$ is continuous and 
%${\rm Gal}(\overline{\mathbb{Q}}/\mathbb{Q})$ is compact and
%$\overline{\mathbb{F}}_\ell^{\times}$ has the discrete topology, 
%the image of $\rho$ is finite. 
Since $\rho$ has a finite image, 
there is a finite abelian extension $F$ over $\mathbb{Q}$ such that 
${\rm image}(\rho) \simeq {\rm Gal}(F/\mathbb{Q})$.
%Therefore ${\rm Ker}(\rho)$ is open and 
%closed set in ${\rm Gal}(\overline{\mathbb{Q}}/\mathbb{Q})$ and 
%there is a finite extension field $F$ of $\mathbb{Q}$ such that 
%${\rm Ker}(\rho) \simeq {\rm Gal}(\overline{\mathbb{Q}}/F)$.
%By the homomorphism theorem, 
%${\rm Gal}(F/\mathbb{Q})\simeq {\rm image}(\rho)$. Because 
%${\rm image}(\rho)$ is an abelian group, $F/\mathbb{Q}$ is abelian.
%By the Kronecker-Weber theorem, there is an integer $N$ such that 
%$F\subset \mathbb{Q}(\zeta_N)$ where $\zeta_N = e^{\frac{2\pi i}{N}}$.
Thus, by Kronecker-Weber theorem, there is a positive integer 
$M$ and we have the factorisation 
$\rho = \rho_\chi = \rho_{\chi,M}\circ \pi_M$, 
where $\rho_{\chi,M}:{\rm Gal}(\mathbb{Q}(\zeta_M)/\mathbb{Q})$ 
$\simeq (\mathbb{Z}/M\mathbb{Z})^\times \stackrel{\chi}{\to}$
$\overline{\mathbb{F}}_\ell^{\times}$
and 
$\pi_M : {\rm Gal}(\overline{\mathbb{Q}}/\mathbb{Q}) \to$
%${\rm Gal}(\overline{\mathbb{Q}}/\mathbb{Q})/{\rm
%Gal}(\overline{\mathbb{Q}}/\mathbb{Q}(\zeta_N))\simeq$ 
${\rm Gal}(\mathbb{Q}(\zeta_M)/\mathbb{Q})$. 
Since $\rho$ is tamely ramified at $\ell$, 
we may take $M = \ell N$ where $N$ is the Artin conductor of 
$\rho$. In particular, $\rho$ comes from a mod $\ell N$ 
character.
%we have $M = \ell N$.
%In these circumstances, 
%$1$-dimensional mod $\ell$ representations with Artin conductor $N$
%come from mod $\ell N$ Dirichlet characters. 

Let $\rho$ and $\rho'$ be two $1$-dimensional mod $\ell$ Galois
 representations of Artin conductors $N$, and $\chi$ and $\chi'$
 be two mod $\ell N$ Dirichlet characters such that 
$\rho=\rho_\chi$ and  $\rho'=\rho_{\chi'}$. 
%Then since $\rho$ and $\rho'$ 
%are tamely ramified at $\ell$, 
%the conductors of $\chi$ and $\chi'$ are $\ell N$.
On the $1$-dimensional case of the problem, we claim that 
there is a positive number $\kappa=\kappa(N,\ell)$ such that 
if $\rho({\rm Frob}_p) = \rho'({\rm Frob}_p)$ in 
$\overline{\mathbb{F}}_\ell^\times$ for every prime number $p$ 
satisfying $(p,\ell N)=1$ and $p\leq \kappa$, 
then $\rho\simeq \rho'$. 
%This is equivalent to find a positive constant $\kappa$ such that 
%if $\rho \not\simeq \rho'$, 
%$\rho({\rm Frob}_p)\neq \rho'({\rm Frob}_p)=$ for a prime $p$ satisfying
%$p \leq \kappa$ and $(p,\ell N)=1$.

First, we discuss a trivial estimate.
If $\rho_\chi \not\simeq \rho_{\chi'}$, $(\chi/\chi')(n)\neq 1$ for an
integer $n$ such that $1 < n <\ell N$. Then there is a prime $p|n$ such
that $(\chi/\chi')(p)\neq 1$.  
Thus we can take $\kappa = \ell N$. It is a trivial estimate for
$\kappa$.
%If $\rho({\rm Frob}_p)=\rho'({\rm Frob}_p)$ for every prime $p$ satisfy
%$p<\ell N$ and $(p,\ell N)$, $\rho \simeq \rho'$.

Using Burgess' estimate for character sums \cite{Burgess}, 
we obtain a better estimate for $\kappa$. 
Let $d$ be a positive integer, $M$ be an integer such that $0<M<d$, and
 $\chi_0$ be a non-trivial mod $d$ Dirichlet character. According to
 Burgess \cite[Th.~2]{Burgess}, 
\begin{eqnarray*}
 |\sum_{n=1}^{M} \chi_0(n)| \ll_{r,\varepsilon}
  M^{1-\frac{1}{r}}d^{\frac{r+1}{4r^2}+\varepsilon} 
\end{eqnarray*}
for every positive integer $r$ and every positive number $\varepsilon$.
This inequality means that 
\begin{eqnarray*}
 |\sum_{n=1}^{M} \chi_0(n)| < c M^{1-\frac{1}{r}}d^{\frac{r+1}{4r^2}+\varepsilon} 
\end{eqnarray*}
for every positive integer $r$ and every positive number 
$\varepsilon$ with a positive constant $c$ 
depending on $r$ and $\varepsilon$.
When $M>c^rd^{\frac{r+1}{4r}+r\varepsilon}$, 
$c M^{1-\frac{1}{r}}d^{\frac{r+1}{4r^2}+\varepsilon}<M$.
Thus $\chi_0(n) \neq 1$ for $0< n < c^rd^{\frac{r+1}{4r}+r\varepsilon}+1$.
In our case, applying Burgess' result with $d=\ell N$ and 
$\chi_0 =\chi/\chi'$,   
we can take $\kappa = c^r(\ell N)^{\frac{r+1}{4r}+r\varepsilon}+1$ with
the above $c,r,\varepsilon$. Thus we obtain the estimate 
\begin{eqnarray*}
 \kappa \ll_{r,\varepsilon} (\ell N)^{\frac{r+1}{4r}+r\varepsilon}.
\end{eqnarray*}

On the estimate of character sums, 
it is conjectured that the bound is some polynomial 
order of the logarithm. 
Indeed, Ankeny \cite[Th.~2]{Ankeny} proved, 
under GRH, the following estimate:
\begin{eqnarray*}
 |\sum_{n=1}^{M} \chi_0(n)|  \ll (\log (\ell N))^2. 
\end{eqnarray*}
Using this, we obtain
\begin{eqnarray*}
 \kappa \ll (\log (\ell N))^2
\end{eqnarray*} 
under GRH.

%It is sufficient to find a minimum number such that 
%all class of $(\mathbb{Z}/\ell N\mathbb{Z})^\times$ has 
%at least one prime number.
%%This is exactly the Linnik's theorem.
%Let $a,d$ be two positive integers such that $(a,d) = 1$. 
%We define 
%\begin{eqnarray*}
% p_{\rm min}(a,d) &=& {\rm min}\{p:{\rm prime} \ | \ p = a + d
%m \ (m \in
%  \mathbb{Z}_{\geq 0}) \}, \\
%p_{\rm min}(d) &=& {\rm max}\{p_{\rm min}(a,d) \ | \ 
%0 < a < d, \ (a,d)= 1 
%%a\in (\mathbb{Z}/d\mathbb{Z})^\times
%  \}.
%\end{eqnarray*}
%%and put $p_{\rm max}(N) = {\rm max}\{p_{\rm min}(a,N) \ | \ a\in (\mathbb{Z}/N\mathbb{Z})^\times  \}$.
%According to Linnik \cite{Linnik-1},\cite{Linnik-2}, there is a positive number $L$ such that 
%\begin{eqnarray*}
%p_{\rm min}(d) \ll d^L
%\end{eqnarray*}
%where the inequality means that there is a positive constant $c$ such that 
%\begin{eqnarray*}
% p_{\rm min}(d) < c d^L.
%\end{eqnarray*} 
%On the assumption of GRH, it can be shown that 
%$p_{\min}(d) \ll \phi(d)^2\log^2 d$.
%Heath-Brown showed that $L=5.5$ with an effectively computable constant 
%$c$ (\cite{H-B}). 
%%(It is also conjectured that $c=1$ and $L=2$.) 
%%Therefore in $1$-dimensional case we can take $\kappa = cN^{5.5}$. 
%Therefore we obtain the following proposition.
%\begin{prop}
% In the $1$-dimensional case, constant $\kappa$ in the problem
%is $c{(\ell N)}^{5.5}$ where $c$ is an effectively
% computable constant 
%%(independent of $\ell$ and $N$) 
%in \cite{H-B}. 
%\end{prop}
%%Remark that $\kappa$ is independent of $\ell$ in this case.\\ \par

\subsection{Two-dimensional case.}
In the $2$-dimensional case, 
we prove the following main result:
\begin{thm}
Let $\ell$ be a prime number and $N$ be a positive integer such that 
$\ell \nmid N$. 
Let $\rho,\rho'
:{\rm Gal}(\overline{\mathbb{Q}}/\mathbb{Q})$
$\to GL_2(\overline{\mathbb{F}}_\ell)$ be two semisimple  
$2$-dimensional mod $\ell$ Galois representations 
%of conductor $N$.
with Artin conductor dividing $N$.
Assume that $\rho$ is odd.
%\begin{eqnarray*}
% \rho: {\rm Gal}(\overline{\mathbb{Q}}/\mathbb{Q})
%  \longrightarrow GL_2(\overline{\mathbb{F}_{\ell}})
%\end{eqnarray*} 
Let 
\begin{eqnarray*}
 \kappa = \kappa(N, \ell) = \left\{\begin{array}{cc}
			     \ds\frac{\ell(\ell^2 - 1)^2}{12}
                NN'\prod_{p|N}\left(1+\ds\frac{1}{p}\right) &  \mathrm{if} \ \ell > 2, \\
		 4NN' \prod_{p|N}\left(1+\ds\frac{1}{p}\right) &
				    \mathrm{if} \ \ell = 2.  \end{array}
			     \right.
% \kappa = \kappa(N, \ell) = \left\{\ds\frac{\ell^2 - 1}{12}
%                [\Gamma_0(1):\Gamma(\ell N)] \ \ \ {\rm if} \ \ell > 2, 
%		\atop
%			  \ds\frac{1}{3}
%                [\Gamma_0(1):\Gamma(2N)\cap\Gamma_1(4)] \ \ \ {\rm if} \ \ell = 2.   \right.
\end{eqnarray*}
where $N'=\prod_{p|N, \  p^2 \nmid N}p$.
If  
\begin{eqnarray*}
% {\rm Tr}(\rho({\rm Frob}_p)) = 
%{\rm Tr}(\rho'({\rm Frob}_p)) \mbox{ in } 
%\overline{\mathbb{F}}_{\ell}, \\
%\det(\rho({\rm Frob}_p)) = 
%\det(\rho'({\rm Frob}_p)) \mbox{ in } 
%\overline{\mathbb{F}}_{\ell}
\det(1- \rho({\rm Frob}_p)T) = \det(1- \rho'({\rm Frob}_p)T) \quad
\mbox{ in $\overline{\mathbb{F}}_\ell[T]$} 
\end{eqnarray*}
for 
every prime number $p$ satisfying $p\leq\kappa$ and $p \nmid \ell N$, $\rho$ is isomorphic to $\rho'$.
%In particular, if  
%\begin{eqnarray*}
% \kappa' = \kappa'(N, \ell) = \left\{
%			     \begin{array}{ll} 
%			      \ds\frac{\ell(\ell^2 - 1)^2}{24}
%			       N^3\prod_{p|N}(1- \frac{1}{p^2}) 	     
%			     &  \quad {\rm if} \ \ell > 2, \\
%			  \ds 6 N^3\prod_{p|N}(1- \frac{1}{p^2})
%			  & \quad {\rm if} \
%			   \ell = 2 \ {\rm and} \ N > 1, \\
%			  \ds 12 &
%			  \quad {\rm if} \ \ell = 2 \ {\rm and} \ N = 1.  
%			     \end{array} 
%\right.
%\end{eqnarray*}
%and 
%\begin{eqnarray*}
%% {\rm Tr}(\rho({\rm Frob}_p)) =
%%{\rm Tr}(\rho'({\rm Frob}_p)) \mbox{ in } 
%%\overline{\mathbb{F}}_{\ell}, \\
%%\det(\rho({\rm Frob}_p)) = 
%%\det(\rho'({\rm Frob}_p)) \mbox{ in } 
%%\overline{\mathbb{F}}_{\ell}
%\det(1- \rho({\rm Frob}_p)T) = \det(1- \rho'({\rm Frob}_p)T) \quad
%\mbox{ in $\overline{\mathbb{F}}_\ell[T]$} 
%\end{eqnarray*}
%%${\rm Tr}(\rho({\rm Frob}_p)) \equiv 
%%{\rm Tr}(\rho'({\rm Frob}_p)) \pmod{\lambda}$ 
%for every prime number $p$ satisfying $p\leq\kappa'$ and $p \nmid \ell N$, 
%then $\rho$ is isomorphic to $\rho'$. \par
\end{thm}

\begin{proof}
First, we prove that $\rho'$ is odd.
By the assumption, 
\begin{eqnarray*}
 \det(\rho({\rm Frob}_p)) =\det(\rho'({\rm Frob}_p)) 
%\quad \mbox{in $\overline{\mathbb{F}}_\ell$} 
\end{eqnarray*}
for every prime $p$ satisfying $p<\ell N <\kappa$ and 
$p\nmid \ell N$. Thus, by the trivial estimate on 
the $1$-dimensional case, 
\begin{eqnarray*}
 \det(\rho) =\det(\rho') 
% \quad \mbox{in $\overline{\mathbb{F}}_\ell$} 
\end{eqnarray*}
holds for every prime number $p$ satisfying $p\nmid \ell N$. 
%By Chebotarev's density theorem, for a complex conjugation 
%$c$, there is a prime $p$ such that $c={\rm Frob}_p$ and 
%$p\nmid \ell N$. Thus, for such $p$, we have 
%\begin{eqnarray*}
%  \det(\rho'(c)) &=&  \det(\rho'({\rm Frob}_p)) \\
%   &=& \det(\rho({\rm Frob}_p)) \\ 
% &=&\det(\rho(c)) = -1, 
%\end{eqnarray*}
%{\it i.e.,} 
Thus, $\rho'$ is also odd.
By Remark \ref{remark:modularity}, semisimple odd continuous 
$2$-dimensional mod $\ell$
 Galois representations come from Hecke eigenforms.
%Thus let $f, g$ be two newforms of weight
% $k(\rho),k(\overline{\rho '})$ of level $\Gamma_0(N)$.  
%corresponding to $\rho, \rho'$.
Because $N(\rho),N(\rho')|N$ and $N|\ell N$, 
we can take the appropriate eigenform $f\in$
$M_{k_1}(\Gamma_0(\ell N),\epsilon)$
(resp. $g\in$
$M_{k_2}(\Gamma_0(\ell N),\epsilon)$)
such that $\rho \simeq \overline{\rho_f}^{ss}$ 
 (resp. $\rho' \simeq \overline{\rho_g}^{ss}$),  
$2 \leq k_1,k_2 \leq \ell^2 - 1$ if $\ell>2$, and $k_1,k_2=2,4$ if
 $\ell=2$. Here $\overline{\rho_f}^{ss}$ is the semisimplification of $\overline{\rho_f}$.
%By the assumption that $N(\rho)|N$ (resp. $N(\rho')|N$), 
%we can regard as $f\in $ 
%$S_{k(\rho)}(\Gamma_0(N),\epsilon(\rho))$ (resp. 
%$g\in$
%$S_{k(\rho')}(\Gamma_0(N),\epsilon(\rho'))$).

Next, 
we apply the operator $\pi$ defined in \S.~$2.1.1$. 
We set $\tilde{f}=\pi(f)(z)\in M_{k_1}(\Gamma_0(\ell^2 NN'),\epsilon)$ and 
$\tilde{g}=\pi(g)(z)\in M_{k_2}(\Gamma_0(\ell62 NN'),\epsilon)$.
%, where $N'=\prod_{p|N}p$.
%Let 
%\begin{eqnarray*}
% \tilde{f}(z) = \sum_{n\geq 1:(n,\ell N)=1} a_n q^n, \\
% \tilde{g}(z) = \sum_{n\geq 1:(n,\ell N)=1} b_n q^n
%\end{eqnarray*}
%be their Fourier expansions. Then 
%\begin{eqnarray*}
% c_n = \left\{
%	       \begin{array}{ll}
%		a_n & \quad \mbox{if $(n,\ell N)=1$}, \\
%	        0    & \quad \mbox{otherwise}, 
%	       \end{array} 
%	\right.\\
%d_n = \left\{
%	       \begin{array}{ll}
%		b_n & \quad \mbox{if $(n,\ell N)=1$}, \\
%	        0    & \quad \mbox{otherwise}.
%	       \end{array}
%       \right.
%\end{eqnarray*}
Let $\tilde{f}(z) = \sum_{n=1}^{\infty}a_n q^n$ and
 $\tilde{g}(z) = \sum_{n=1}^{\infty}b_n q^n$ be their Fourier expansions.
Remark that $a_n=b_n=0$ for all $n$ such that $(n,\ell N)>1$ and 
%$a_{mn}=a_m a_n$ and $b_{mn}=b_mb_n$ for all $m,n\in \mathbb{Z}_{>0}$.
\begin{eqnarray*}
 &&\left\{a_{mn}=a_m a_n \atop b_{mn}=b_mb_n \right. \quad \mbox{ for all $(m,n)=1$}, \\
 &&\left\{ a_{p^n} = a_{p^n-1}a_p +
  \epsilon(p)p^{k_1-1}a_{p^{n-2}} \atop
b_{p^n} = b_{p^n-1}b_p +
  \epsilon(p)p^{k_2-1}b_{p^{n-2}}
  \right. 
%\quad \mbox{ for every prime number $p$ such that 
%$p\nmid N$}.
\end{eqnarray*}
for every prime number $p$ such that $p\nmid \ell N$.
Let $K_f = \mathbb{Q}(\dots, a_n, \dots, \epsilon)$ 
(resp. $K_g = \mathbb{Q}(\dots, b_n, \dots,\epsilon)$) be the field generated by
 the all Fourier coefficients of $f$ (resp. $g$) and the values of $\epsilon$, and $\mathcal{O}_{K_f}$
 (resp. $\mathcal{O}_{K_g}$) 
be the ring of integers of $K_f$ (resp. $K_g$).
Let $\lambda_f$ (resp. $\lambda_g$) be a maximal ideal in
 $\mathcal{O}_{K_f}$ (resp. $\mathcal{O}_{K_g}$) such that
 $\lambda_f | \ell\mathcal{O}_{K_f}$ 
(resp.  $\lambda_g | \ell\mathcal{O}_{K_g}$) and 
$\mathbb{F}_f =\mathcal{O}_{K_{f,\lambda_f}}/\lambda_f$ 
(resp. $\mathbb{F}_g =\mathcal{O}_{K_{f,\lambda_g}}/\lambda_g$).
%, 
%$\lambda_g$ be a maximal ideal in $\mathcal{O}_{K_g}$ such that
% $\lambda_g | \ell\mathcal{O}_{K_g}$,
%and 
%$\mathbb{F}_g =\mathcal{O}_{K_{g,\lambda_g}}/\lambda_g$.
Then 
\begin{eqnarray*}
 \det(1- \rho({\rm Frob}_p)T) &=& 
1 - a_pT + \epsilon(p)p^{k_1 -1}T^2 \quad 
%\mod{\lambda_f},\\
\mbox{in $\mathbb{F}_f[T]$},\\
\det(1- \rho'({\rm Frob}_p)T) &=& 
1 - b_pT + \epsilon(p)p^{k_2-1}T^2 
%\quad  \mod{\lambda_g}
\quad \mbox{in $\mathbb{F}_g[T]$}
\end{eqnarray*}
for every prime $p$ such that $p\nmid \ell N$.
Let $L$ 
%\begin{eqnarray*}
% L = K_f K_g
%\end{eqnarray*}
be the Galois closure of the composite of $K_f$ and $K_g$, 
$\mathcal{O}_L$ be the ring of the integers of 
 $L$, and $\lambda$ be a maximal ideal in $\mathcal{O}_L$ 
such that
 $\lambda|\lambda_f \mathcal{O}_L$ and 
$\lambda|\lambda_g \mathcal{O}_L$, and
$\mathbb{F} =\mathcal{O}_{L,\lambda}/\lambda$.
%(we can take $\lambda$ as a 
%discrete value in $L\otimes_\mathbb{Q}\mathbb{Q}_\ell$). 
%(the existence is proved 
%by choose and fix an embedding 
%$\overline{\mathbb{Q}} \hookrightarrow $
%$\overline{\mathbb{Q}}_\ell$).
Then 
\begin{eqnarray*}
 \det(1- \rho({\rm Frob}_p)T) &=& 
1 - a_pT + \epsilon(p)p^{k_1-1}T^2 
%\quad \mod{\lambda},\\
\quad \mbox{in $\mathbb{F}[T]$}, \\
%\mbox{in $\mathcal{O}_{K_f}/\lambda_f[T]$},\\
\det(1- \rho'({\rm Frob}_p)T) &=& 
1 - b_pT + \epsilon(p)p^{k_2-1}T^2 
%\quad  \mod{\lambda}
\quad \mbox{in $\mathbb{F}[T]$}
\end{eqnarray*}
for all prime $p$ such that $p\nmid \ell N$.
%On the first assumption, we have
%\begin{eqnarray*}
%  a_p \equiv b_p \pmod{\lambda}
%\end{eqnarray*}
%for every prime $p$ such that $p\nmid \ell N$ and $p\leq\kappa$ 
%as in the Theorem $1$.
\par

On the assumption, we have
\begin{eqnarray*}
  a_p \equiv b_p \pmod{\lambda}
\end{eqnarray*}
for every prime $p$ such that $p\nmid \ell N$ and 
$p\leq\kappa$ (for $\kappa$ in Theorem $1$).
Because $a_{mn} = a_ma_n, b_{mn}= b_mb_n$ for $(m,n)=1$ and 
$a_{p^n} \equiv b_{p^n} \pmod{\lambda}$ for every prime number $p$ such that 
$p\leq \kappa$ (indeed, for such prime $p$,   
\begin{eqnarray*}
 a_{p^2} &=& a_p^2 + \epsilon(p)p^{k_1-1}a_1 \\
 &\equiv& a_p^2 + \det(\rho({\rm Frob}_p)) \pmod{\lambda} \\
 &\equiv& b_p^2 +  \det(\rho'({\rm Frob}_p)) \pmod{\lambda} \\
 &\equiv& b_p^2 + \epsilon(p)p^{k_2-1}b_1  \pmod{\lambda} \\
 &=& b_{p^2} 
\end{eqnarray*}
and by the induction), we have 
\begin{eqnarray*}
 a_n \equiv b_n \pmod{\lambda}
\end{eqnarray*}{\large }
for every $n$ such that $n \leq \kappa$. 
While it is easy to check that 
\begin{eqnarray*}
 \kappa = \left\{\begin{array}{cc}
	   \ds\frac{\ell^2 - 1}{12}[\Gamma_0(1): \Gamma_0(\ell^2 NN')\cap
	    \Gamma_1(\ell)] & \mathrm{if \ } \ell >2, \\ 
		  \ds\frac{4}{12}[\Gamma_0(1): \Gamma_0(4 NN')\cap
		   \Gamma_1(4)] & \mathrm{if \ } \ell =2.
		 \end{array}\right. 
\end{eqnarray*}
%By Sturm's and Kohnen's result 
By Lemma $1$ and Lemma $2$ 
%(remark that $\Gamma(\ell N)\cap \Gamma_1(\ell)= \Gamma(\ell N)$)
and the fact that 
$2 \leq k(\rho),k(\rho') \leq \ell^2 - 1$ if $\ell >2$ and
 $k(\rho),k(\rho') =2$ or $4$ if $\ell =2$, 
we have $\tilde{f} \equiv \tilde{g} \pmod{\lambda}$. 
It means that $a_n \equiv b_n \pmod{\lambda}$ for all $n$. 
Therefore ${\rm Tr}(\rho({\rm Frob}_p))  \equiv $ 
${\rm Tr}(\rho'({\rm Frob}_p))\pmod{\lambda}$  for every prime $p$ such that 
$p \nmid \ell N$.
%While for every prime number $p \nmid \ell N$, 
%\begin{eqnarray*}
% \det(\rho({\rm Frob}_p)) &\equiv&
%  \epsilon(\rho)(p)p^{k(\rho)-1} \mod{\lambda} \\
% &=& a_{p^2} - a_p^2 \\
% &\equiv& b_{p^2} - b_p^2 \mod{\lambda} \\
% &=& \epsilon(\rho')(p)p^{k(\rho')-1} \\
% &\equiv& \det(\rho'({\rm Frob}_p)) \mod{\lambda}. 
%\end{eqnarray*} 
By Lemma $3$, $\rho$ is isomorphic to $\rho'$. 
%
%On the second assertion, if $\ell >2$, 
%\begin{eqnarray*}
% [\Gamma_0(1):\Gamma_0(\ell^2 NN')] = \ds\frac{1}{2}(\ell N)^3\prod_{p|\ell N} (1-\frac{1}{p^2}).
%\end{eqnarray*}
%Thus we have the first equation.
%If $\ell = 2$, 
%\begin{eqnarray*}
% [\Gamma_0(1):\Gamma(2N)\cap\Gamma_1(4)] &\leq &
%  [\Gamma_0(1):\Gamma(2N)][\Gamma_0(1):\Gamma_1(4)] \\
% &=& 6[\Gamma_0(1):\Gamma(2N)] \\ 
% &=& \left\{
%      \begin{array}{ll}
%       \ds 18N^3\prod_{p|N}(1-\frac{1}{p^2}) & \quad  \mbox{if $N>1$}, \\
%	36 & \quad \mbox{if $N=1$}.
%       \end{array}
%\right.
%\end{eqnarray*}
%Therefore we obtain the other equations. 
\end{proof}

By Theorem 1, we have estimate
\begin{eqnarray*}
 \kappa \ll \ell^5N^2\log{N}.
\end{eqnarray*}
Comparing with the estimate under GRH 
on the $1$-dimensional 
case in section 3.1, it is clear the above estimate 
is very large. 
%In particular, we have seen that, on the $1$-dimensional case, 
%under the assumption of GRH, the following estimate 
%\begin{eqnarray*}
% \kappa \ll (\log(\ell N))^2.
%\end{eqnarray*}
We guess that the estimate for $\kappa$ 
can be improved in the $2$-dimensional case. 
In general, we guess that we can take 
some polynomial of $\log{\ell N}$ as 
the upper bound for $\kappa$
in the arbitrary dimensional case. 
%Thus we conjecture as follows:
%Thus it is natural that we conjecture as follows:
Thus we ask the following question: 
\begin{question}
 Let $n$ be a positive integer, 
and $\kappa$ be the positive number in the problem 
of the $n$-dimensional case. 
Then can we take 
\begin{eqnarray*}
 \kappa =(\log{\ell N})^d
\end{eqnarray*}
for some positive integer $d$?
%holds for some positive integer $n$. 
\end{question}



%% References


%\bibliographystyle{elsarticle-num}






%\bibliographystyle{amsplain}
\bibliographystyle{alpha}
\bibliography{takai}

\begin{thebibliography}{Koh04}

\bibitem[Ank52]{Ankeny}
N.C. Ankeny.
\newblock The least quadratic non residue.
\newblock {\em Ann. of Math. (2)}, 55:65--72, 1952.

\bibitem[Bur63]{Burgess}
D.A. Burgess.
\newblock On character sums and {$L$}-series. {II}.
\newblock {\em Proc. London Math. Soc. (3)}, 13:524--536, 1963.

\bibitem[DS74]{Deligne-Serre}
P.~Deligne and J.-P. Serre.
\newblock Formes modulaires de poids $1$.
\newblock {\em Ann. Sci. \'Ecole Norm. Sup., (4)}, 7:507--530, 1974.

\bibitem[DS05]{Diamond-Shurman}
F.~Diamond and J.~Shurman.
\newblock {\em A first course in modular forms}.
\newblock Number 228 in Graduate Texts in Mathematics. Springer-Verlag, New
  York, 2005.

\bibitem[Edi92]{edixhoven1992weight}
B.~Edixhoven.
\newblock {The weight in Serre's conjectures on modular forms}.
\newblock {\em Inventiones mathematicae}, 109(1):563--594, 1992.

\bibitem[Koh04]{Kohnen}
W.~Kohnen.
\newblock On {F}ourier coefficients of modular forms of different weights.
\newblock {\em Acta Arith.}, 113(1):57--67, 2004.

\bibitem[KW]{K-W}
C.~Khare and J-P. Wintenberger.
\newblock Serre's modularity conjecture {I} and {II}.
\newblock available at http://www.math.utah.edu/~shekhar/papers.html.

\bibitem[Miy89]{Miyake}
T.~Miyake.
\newblock {\em Modular forms}.
\newblock Springer-Verlag, Berlin, 1989.

\bibitem[Ser87]{Serre}
J.-P. Serre.
\newblock Sur les repr{\'e}sentations modulaires de degr{\'e} 2 de {${\rm
  Gal}(\overline{\mathbb{Q}}/\mathbb{Q})$}.
\newblock {\em Duke Math. J.}, 54(1):179--230, 1987.

\bibitem[Shi71]{Shimura}
G.~Shimura.
\newblock {\em Introduction to the arithmetic theory of automorphic forms}.
\newblock Princeton University Press, Princeton, NJ, 1971.

\bibitem[Stu87]{Sturm}
J.~Sturm.
\newblock On the congruence of modular forms.
\newblock In {\em Number theory}, volume 1240 of {\em Lecture Notes in Math.},
  pages 275--280, New York, 1984-1985, 1987. Springer, Berlin.

\end{thebibliography}

\end{document}